\documentclass[article]{elsarticle}
\RequirePackage{setspace}
\usepackage{graphicx}
\usepackage{fonttable}
\usepackage{amsthm}
\usepackage{amssymb}
\usepackage{amsmath}
\usepackage{amssymb}
\usepackage{mhequ}
\usepackage{color}
\usepackage{hyperref}
\usepackage{float}
\usepackage{verbatim}
\usepackage[normalem]{ulem}

\usepackage{subcaption}
\usepackage[normalem]{ulem} 
\newcommand{\RMSE}{$\tiny {\rm RMSE}\,\,$}
 
 \theoremstyle{plain}

\numberwithin{equation}{section}
\numberwithin{figure}{section}
\newtheorem{thm}{Theorem}[section]
\newtheorem{lemma}[thm]{Lemma}
\newtheorem{properties}[thm]{Property}
\newtheorem{pr}[thm]{Proposition}

\newtheorem{remark}[thm]{Remark}

\usepackage{amsfonts}
\usepackage{amsthm}

\newcommand{\bbR}{\mathbb{R}}

\newcommand{\rev}[1]{#1}

\newcommand{\as}[1]{#1}

\definecolor{darkred}{rgb}{.7,0,0}

\definecolor{darkgreen}{rgb}{0.4,0.7,0}

\definecolor{darkblue}{rgb}{0,0,0.7}

\begin{document}
\begin{frontmatter}

\title
{Filter Accuracy For the \rev{Lorenz 96 Model: }
Fixed Versus Adaptive Observation Operators}

\author{K.J.H. Law\fnref{myfootnote1}}

\author{D. Sanz-Alonso\fnref{myfootnote2}}

\author{A. Shukla\fnref{myfootnote2}}
\ead{a.shukla@warwick.ac.uk}
\author{A.M. Stuart\fnref{myfootnote2}}

\fntext[myfootnote1]{Computer Science and Mathematics Division, 
Oak Ridge National Laboratory, 
Oak Ridge, TN, 37831, USA.
}
\fntext[myfootnote2]{Mathematics Institute, University of Warwick, Coventry CV4 7AL, UK.}

\begin{abstract}

In the context of filtering chaotic dynamical systems it is well-known that partial observations, if sufficiently informative, can be used to control the inherent uncertainty due to chaos. The purpose of this paper is to investigate, both theoretically and numerically, conditions on the observations of chaotic systems under which they can be accurately filtered. In particular, we highlight the advantage of adaptive observation operators over fixed ones. The Lorenz '96 model is used to exemplify our findings. 

We consider discrete-time and continuous-time observations in our theoretical developments. We prove that, for fixed observation operator, the 3DVAR filter can recover the system state within a neighbourhood determined by the size of the observational noise. It is required that a sufficiently large proportion of the state vector is observed, and an explicit form for such sufficient fixed observation operator is given. Numerical experiments, where the data is incorporated by use of the 3DVAR and extended Kalman filters, suggest that less informative fixed operators than given by our theory can still lead to accurate signal reconstruction. Adaptive observation operators are then studied numerically; we show that, for carefully chosen adaptive observation operators, the proportion of the state vector that needs to be observed is drastically smaller than with a fixed observation operator. Indeed, we show that the number of state coordinates that need to be observed may even be significantly smaller than the total number of positive Lyapunov exponents of the underlying system.

\end{abstract}


\end{frontmatter}

\section{Introduction}
Data assimilation is concerned with the blending of data and dynamical
mathematical models, often in an online fashion where it is known as
filtering; motivation comes from
applications in the geophysical sciences such as weather forecasting 
\cite{kal03}, oceanography \cite{ben02} and oil reservoir 
simulation \cite{orl08}. Over the last decade there has been a growing
body of theoretical understanding which enables use of the theory of
synchronization in dynamical systems to establish desirable properties
of these filters. This idea is highlighted in the recent book
\cite{hdia13} from a physics perspective and, on the rigorous mathematical
side, has been developed from a pair of papers by Olson, Titi and co-workers
\cite{ho03,hayden2011discrete}, in the context of the Navier-Stokes equation in which
a finite number of Fourier modes are observed.
This mathematical work of Olson and Titi concerns
perfect (noise-free) observations, but the ideas have been extended
to the incorporation of noisy data for the Navier-Stokes equation in the papers \cite{DLSZ12,brett2012accuracy}.
Furthermore the techniques used are quite robust to different
dissipative dynamical systems, and have been demonstrated to
apply in the Lorenz '63 model \cite{hayden2011discrete,lawshuklastuart}, and also to point-wise in space and
continuous time observations \cite{aot14} by use of a control theory perspective
similar to that which arises from the derivation of continuous
time limits of discrete time filters \cite{DLSZ12}.
A key question in the field is to determine relationships
between the underlying dynamical system and the observation
operator which are sufficient to ensure that the signal can be
accurately recovered from a chaotic dynamical system, whose
initialization is not known precisely, by the use of observed
data. Our purpose is to investigate this question
theoretically and computationally. We work in the context of 
the Lorenz '96 model, widely adopted as a useful test model in 
the atmospheric sciences data assimilation 
community {\cite{majda2012filtering, TELA:TELA076}}.

The primary contributions of the paper are:
(i) to theoretically demonstrate the robustness 
of the methodology proposed by Olson and Titi, by extending 
it to the Lorenz '96 model; (ii) to highlight the gap between such theories and what can be achieved in practice, by performing careful numerical experiments;
and (iii) to illustrate the power of
allowing  the observation operator to adapt to the dynamics as
this leads to accurate reconstruction of the signal based on very 
sparse observations.  Indeed our approach in (iii) suggests highly 
efficient new algorithms where the observation operator is allowed 
to adapt to the current state of the dynamical system. The question of 
how to optimize the observation operator
to maximize information was first addressed 
in the context of atmospheric science applications in \cite{lorenzemanuel}. 
{The adaptive observation operators that we propose
are not currently practical for operational atmospheric data
assimilation, but they suggest a key principle which should
underlie the construction of adaptive observation operators: to learn
as much as possible about modes of instability in the dynamics
at minimal cost. }

The outline of the  paper is as follows. In section \ref{sec:S} we introduce the model set up and 
a family of Kalman-based filtering schemes which include as particular cases the Three-dimensional Variational method (3DVAR) and the Extended Kalman Filter (ExKF) used in this paper.
All of these methods may be derived
from sequential application of a minimization principle which
encodes the trade-off between matching the model and matching the data.
In section \ref{l96} we describe the Lorenz '96 model
and discuss its properties that are
relevant to this work. In section \ref{sec:C} we introduce a fixed observation operator which corresponds to observing two thirds of the signal and study theoretical properties of the 3DVAR  filter, in both a continuous 
and a discrete time setting.
In section \ref{sec:N} we introduce an adaptive observation operator which employs knowledge of 
the linearized dynamics over the assimilation window to 
ensure that the unstable directions of the dynamics are observed. We then numerically 
study the performance of a range of filters using the adaptive observations. 
In 
subsection \ref{ssec:3} we consider the 3DVAR  method, 
whilst subsection \ref{ssec:EX} focuses on the Extended Kalman Filter (ExKF). 
In subsection \ref{ssec:EX} we also compare the adaptive observation 
implementation of the ExKF with the AUS scheme \cite{trevisan04}
which motivates our work. The AUS scheme projects the model covariances into the subspaces governed by the unstable dynamics, whereas we use this idea on the observation operators themselves, rather than on the covariances.
In section \ref{sec:CO} we summarize the work and draw some brief conclusions.
In order to maintain a readable flow of ideas,
the proofs of all properties, propositions and theorems stated in the
main body of the text are collected in an appendix.

Throughout the paper we denote by $\langle\cdot,\cdot\rangle$
and $|\cdot|$  the standard Euclidean inner-product
and norm. For positive-definite matrix $C$ we define 
$|\cdot|_{C}:=|C^{-\frac12}\cdot|.$

\section{Set Up}\label{sec:S}
We consider the ordinary differential equation (ODE) 
\begin{equation}\label{mthclf}
\frac{\textrm{d}v}{\textrm{d}t}=\mathcal{F}(v),\quad v(0)=v_0,
\end{equation}
where the solution to \eqref{mthclf} is referred to as the {\em signal}. 
We denote by $\Psi: \bbR^J \times \bbR^+ \to \bbR^J$  the solution operator for the equation
\eqref{mthclf}, so that $v(t)=\Psi(v_0;t).$ In our discrete time filtering developments we assume that, for some fixed $h>0,$ the signal is subject
to observations at times $t_k:=kh,\,k\ge 1.$ We then write 
$\Psi(\cdot):=\Psi(\cdot;h)$ and $v_k:=v(kh)$, with slight abuse of notation to simplify the presentation.
Our main interest is in using partial observations of the discrete time dynamical system
\begin{equation}
\label{eq:ds}
v_{k+1}=\Psi(v_k), \quad k \ge 0,
\end{equation}
to make estimates of the state of the system. To this end we introduce the family of linear
observation operators $\{H_k\}_{k \ge 1}$, where 
$H_k: \mathbb{R}^J \to {\mathbb{R}^{M}}$ 
is assumed to have rank {(which may change with $k$) less than or equal to} $M\le J$. We then consider data $\{y_k\}_{k \ge 1}$ given by
\begin{equation}
\label{eq:obs}
y_{k}=H_k v_k+\nu_k, \quad k \ge 1,
\end{equation}
where we assume that the random and/or systematic error $\nu_k$ (and hence also $y_k$) is contained in $\bbR^M$ 
If $Y_k=\{y_{\ell}\}_{\ell=1}^k$ then the objective of filtering
is to estimate $v_k$ from $Y_k$ given incomplete knowledge of $v_0$;
furthermore this is to be done in a sequential fashion, using
the estimate of $v_k$ from $Y_k$ to determine the estimate of $v_{k+1}$ 
from $Y_{k+1}.$
We are most interested in the case where $M<J$, so that the observations
are partial, and $H_k\bbR^J$ is a strict {$M$ dimensional} subset of $\bbR^J$; in 
particular we address the question of how small $M$ can be chosen whilst still allowing accurate recovery of the signal over long time-intervals.

Let $m_k$ denote our estimate of $v_k$ given $Y_k$. The discrete time filters 
used in this paper have the form
\begin{align}
m_{k+1}&={\rm argmin}_m\left\{\frac{1}{2}\bigl|m-\Psi(m_k)\bigr|_{\widehat{C}_{k+1}}^2+\frac{1}{2}\bigl|y_{k+1}-H_{k+1} m\bigr|_\Gamma^2\right\}.
\label{eq:min}
\end{align}
The norm in the second term is only applied within the $M$-dimensional
image space of $H_{k+1}$, where $y_{k+1}$ lies; 
then $\Gamma$ is realized as a positive-definite $M \times M$ matrix in
this image space, and $\widehat{C}_{k+1}$ is a positive-definite $J \times J$ matrix.
The minimization represents a compromise between respecting the model and
respecting the data, with the covariance weights $\widehat{C}_{k+1}$ and $\Gamma$ 
determining the relative size of the two contributions; 
see \cite{LSZ14} for more details. 
Different choices of $\widehat{C}_{k+1}$ give different filtering methods. For instance, the choice  $\widehat{C}_{k+1}=C_0$ (constant in $k$) corresponds to the 3DVAR method. More sophisticated algorithms, such as the ExKF, allow
$\widehat{C}_{k+1}$ to depend on $m_k.$

All the discrete time algorithms we consider proceed iteratively 
 in the sense that the estimate $m_{k+1}$ is determined
by the previous one, $m_k,$ and the observed data $y_{k+1}$; we
are given an initial condition $m_0$ which is an imperfect estimate of $v_0$. 
It is convenient to see the update $m_k\mapsto m_{k+1}$ as a two-step process. 
In the first one, known as the {\em forecast step}, the estimate $m_k$ is evolved with the dynamics of the underlying model yielding a prediction $\Psi(m_k)$
for the current state of the system. In the second step, known as 
the {\em analysis step}, the forecast is used in conjunction with the observed data $y_{k+1}$ to produce the estimate $m_{k+1}$ of the true state of the underlying system $v_{k+1}$, using the minimization principle \eqref{eq:min}.

 In section \ref{sec:C} we study the continuous time filtering problem for fixed observation operator, where the goal is to estimate the value of a continuous time signal
 $$v(t)=\Psi(v_0,t), \quad t\ge 0,$$
at time $T>0.$ As in the discrete case, it is assumed that only incomplete knowledge of $v_0$ is available. In order to estimate $v(T)$ we assume that we have 
access, at each time $0<t\le T,$ to a (perhaps noisily perturbed) projection of the signal given by a fixed, constant in time, observation matrix $H.$
The continuous time
limit of 3DVAR with constant observation operator $H$, is obtained by setting $\Gamma=h^{-1}\Gamma_0$ 
and $\widehat{C}_{k+1}=C$
and letting $h \to 0.$ The resulting filter, derived in \cite{DLSZ12}, is given by
\begin{equation}\label{c3dv}
\frac{\textrm{d}m}{\textrm{d}t}=\mathcal{F}(m)+CH^*\Gamma_0^{-1}\Big(\frac{\textrm{d}z}{\textrm{d}t}-Hm\Big),
\end{equation}
where the observed data is now $z$ -- formally the time-integral of 
the natural continuous time limit of $y$  --
which  satisfies the stochastic differential equation (SDE)
\begin{equation}\label{c3dv1}
\frac{\textrm{d}z}{\textrm{d}t}=H{v}+H\Gamma_0^{\frac{1}{2}}\frac{\textrm{d}w}{\textrm{d}t},
\end{equation}
for $w$ a unit Wiener process.
This filter has the effect of nudging the solution towards the observed data
in the $H$-projected direction. A similar
idea is used in \cite{aot14} to assimilate pointwise
observations of the Navier-Stokes equation.

For the discrete and continuous time filtering schemes as described
we address the following questions: 
\begin{itemize}
\item[--] how does the filter error $\vert m_k-v_k \vert$ behave as $k\rightarrow \infty$ 
(discrete setting)? 
\item[--]how does the filter error $\vert m(t)-v(t) \vert$ behave as $t\rightarrow \infty$
(continuous setting)?
\end{itemize}
We answer these questions in the section \ref{sec:C} in the context of
the Lorenz '96 model: for a carefully chosen fixed observation operator
we determine conditions under which the large time filter error
is small -- this is filter accuracy. 
We then turn to the adaptive observation operator and focus on the following 
lines of enquiry:
\begin{itemize}
\item[--] how much do we need to observe to obtain filter accuracy? (in other 
words what is the minimum rank of the observation operator required?) 

\item[--] how does adapting the observation operator affect the answer to this
question?
\end{itemize}

We study both these questions numerically in  section \ref{sec:N}, again
focussing on the Lorenz '96 model to illustrate ideas.

\section{Lorenz '96 Model}\label{l96}

The Lorenz '96 model is a lattice-periodic system of coupled nonlinear ODE
whose solution $u=(u^{(1)},\ldots,u^{(J)})^T\in\mathbb{R}^J$ satisfies
\begin{equation}\label{1}
\frac{\textrm{d}u^{(j)}}{\textrm{d}t}= u^{(j-1)}(u^{(j+1)}-u^{(j-2)})-u^{(j)}+F\quad\textrm{for}\;j=1,2,\cdots,J, 
\end{equation}
subject to the periodic boundary conditions
\begin{equation}
\label{eq:bcs}
u^{(0)}=u^{(J)},\quad\quad u^{(J+1)}=u^{(1)},\quad\quad u^{(-1)}=u^{(J-1)}.
\end{equation}
Here $F$ is a forcing parameter, constant in time. 
For our numerical experiments we will choose $F$ 
so that the dynamical system exhibits sensitive 
dependence on initial conditions and positive Lyapunov exponents.
For  example, for $F=8$ and $J=60$ the system is chaotic. 
Our theoretical results apply to any choice of the parameter $F$ and 
{to arbitrarily large system dimension $J$.}

It is helpful to write the model in the following form, widely adopted
in the analysis of geophysical models as dissipative dynamical systems \cite{book:Temam1997}:
\begin{equation}
\frac{\textrm{d}u}{\textrm{d}t}+Au+B(u,u)=f,\quad u(0)=u_0\label{0}
\end{equation}
where 
\[ A=I_{J\times J}, \quad f=\left( \begin{array}{c}
F\\
\vdots\\
F
\end{array}\right)_{J\times 1}\]
and for $u, \tilde{u} \in \bbR^J$
\begin{equation*}
B(u,\tilde{u})=-\frac{1}{2}\left(\begin{array}{c}
\tilde{u}^{(2)}u^{(J)}+u^{(2)}\tilde{u}^{(J)}-\tilde{u}^{(J)}u^{(J-1)}-u^{(J)}\tilde{u}^{(J-1)}\\
\vdots\\
\tilde{u}^{(j-1)}u^{(j+1)}+u^{(j-1)}\tilde{u}^{(j+1)}-\tilde{u}^{(j-2)}u^{(j-1)}-u^{(j-2)}\tilde{u}^{(j-1)}\\
\vdots\\
\tilde{u}^{(J-1)}u^{(1)}+u^{(J-1)}\tilde{u}^{(1)}-\tilde{u}^{(J-2)}u^{(J-1)}-u^{(J-2)}\tilde{u}^{(J-1)}\\
\end{array}\right)_{J\times 1}.
\end{equation*}

We will use the following properties of $A$ and $B$, proved in the Appendix:

\vspace{0.1in}

\begin{properties} For $u, \tilde{u} \in \bbR^J$
\label{p:2.1}
\mbox{}
\begin{enumerate}
 \renewcommand{\theenumi}{\arabic{enumi}}
\item $\langle Au,u\rangle= {|u|}^2$.  \label{itm:a}
\item $\langle B(u,u),u\rangle=0$.\label{itm:b}
\item $B(u,\tilde{u})=B(\tilde{u},u)$.\label{itm:c}
\item $|B(u,\tilde{u})|\leq 2|u||\tilde{u}|$.\label{itm:d}
\item $2\langle B(u,\tilde{u}),u\rangle=-\langle B(u,u),\tilde{u}\rangle.$ \label{itm:f} 
\end{enumerate}
\end{properties}

Property (1) shows that the linear term induces dissipation in the model, whilst
property (2) shows that the nonlinear term is energy-conserving.
Balancing these two properties against the injection of energy through
$f$ gives the existence of an absorbing, forward-invariant ball for equation 
\eqref{0}, as stated in the following proposition,
proved in the Appendix. 

\vspace{0.1in}

\begin{pr}\label{absorbing ball}
Let $K=2JF^2$ and define $\mathcal{B}:=\{u\in \mathbb{R}^{J}: |u|^2\le K\}.$ Then $\mathcal{B}$ is an absorbing, forward-invariant ball for equation (\ref{0}):
for any $u_0 \in \mathbb{R}^{J}$ there is time $T=T(|u_0|) \ge 0$ such
that $u(t) \in \mathcal{B}$ for all $t \ge T.$
\end{pr}

\section{Fixed Observation Operator}\label{sec:C}

In this section we consider filtering the Lorenz '96 model
with a specific choice of fixed observation matrix $P$ 
(thus $H_k=H=P$) that we now introduce. First, we
let $\{e_j\}^J_{j=1}$ be the standard basis for the Euclidean space $\mathbb{R}^J$ and assume that $J=3J'$ for some $J'\ge1.$
Then the projection matrix $P$ is defined by replacing every third column of the identity matrix $I_{J\times J}$ by the zero vector: 

\begin{equation} P=\left( \begin{array}{c}\label{P}
e_1,\, e_2,\,  0,\,  e_4,\,  e_5,\,  0, \ldots 
\end{array}\right)_{J\times J}.\end{equation}

Thus $P$ has rank $M=2J'.$ We also define its complement $Q$ as 
\[
Q=I_{J\times J}-P.
\]

\begin{remark} 
Note that in the definition of the projection matrix $P$ we could have chosen 
either the first or the second column to be set to zero periodically, instead of 
choosing every third column this way; the theoretical results in the
the remainder of this section would be unaltered by doing this.
\end{remark}
\rev{\begin{remark}
For convenience of analysis we consider the projection operator $H_k=H=P$ as mapping $\bbR^J \to \bbR^J$, and define the observed data $\{y_k\}_{k \ge 1}$ by
\begin{equation}
\label{eq:obs}
y_{k}= P \big( v_k+\nu_k \big), \quad k \ge 1.
\end{equation}
Thus,  the observations $y_k$ and the noise $P\nu_k$ live in $\bbR^J$ but  have at most $M=2J'$ non-zero entries.
\end{remark}}
The matrix $P$ provides sufficiently rich observations to allow the accurate recovery of the signal in the long-time asymptotic regime, both in continuous and discrete time settings. 
The following property of $P$, proved in the appendix,
plays a key role in the analysis:

\vspace{0.1in}

\begin{properties}
\label{p:2.2}
The bilinear form $B(\cdot,\cdot)$ as defined after \eqref{0} satisfies $B(Qu,Qu)=0$ and, furthermore,
there is a constant $c>0$ such that  
$$|\langle B(u,u),\tilde{u}\rangle|\leq c|u||\tilde{u}||Pu|.$$
\end{properties}

All proofs in the following subsections are given in the Appendix.
{\begin{remark}
Note that the results which follow require that the bilinear form B satisfies Properties \ref{p:2.1}
and \ref{p:2.2}. 
While Properties \ref{p:2.1} 
are shared by the prototypical advection operators, for example Lorenz '63 \cite{lawshuklastuart} and Navier Stokes \cite{brett2012accuracy}, Property \ref{p:2.2} 
needs to be verified independently for each given case, and choice of operator Q.  This property is key to closing the arguments below.
\end{remark}}

\subsection{Continuous Assimilation}
In this subsection we assume that  
the data arrives continuously in time. Subsection \ref{ssec:NOC} deals with noiseless data, and the more realistic noisy scenario is studied in subsection \ref{continuous3dvarsection}. We aim to show that, in the large time asymptotic, the filter is close to the truth.
In the absence of noise our results are analogous to those for the partially observed Lorenz '63 and Navier-Stokes models in \cite{ho03};
in the presence of noise the results are similar to those proved in
\cite{DLSZ12} for the Navier-Stokes equation and in \cite{lawshuklastuart} for the Lorenz '63 model, and generalize the work in \cite{tarn} to non-globally Lipschitz vector fields.

\subsubsection{Noiseless Observations}\label{ssec:NOC}
The true solution $v$ satisfies the following equation
\begin{equation}\label{eq:en}
\frac{\textrm{d}v}{\textrm{d}t}+v+B(v,v)=f,\quad v(0)=v_0.
\end{equation}
Suppose that the projection $Pv$ of the true solution is perfectly observed 
and continuously assimilated into the approximate solution $m.$
The {\em synchronization filter} $m$ has the following form:
\begin{equation}\label{eq:cm}
m=Pv+q,
\end{equation}
where $v$ is the true solution given by \eqref{eq:en} and $q$ satisfies the equation \eqref{0} projected by $Q$ to obtain
\begin{equation}\label{eq:cm1}
\frac{\textrm{d}q}{\textrm{d}t}+q+QB(Pv+q,Pv+q)=Qf,\quad q(0)=q_0.
\end{equation}
Equations \eqref{eq:cm} and \eqref{eq:cm1} form the continuous time synchronization filter. The following theorem shows that the approximate solution converges to the true solution asymptotically as $t\rightarrow \infty.$

\vspace{0.1in}

\begin{thm}
\label{t:54}
Let $m$ be given by the equations \eqref{eq:cm}, \eqref{eq:cm1} and let $v$ 
be the solution of the  
equation \eqref{eq:en} with initial data $v_0 \in {\mathcal B}$,
the absorbing ball in Proposition \ref{absorbing ball}, so that $\underset{t \ge 0}{\sup}|v(t)|^2
\le K.$ Then 
$$\underset{t \to \infty}{\rm lim}{|m(t)-v(t)|}^2=0.$$

\end{thm}

The result establishes that in the case of high frequency in time
observations the approximate solution converges to the true solution 
even though the signal is observed partially at frequency $2/3$ in space.
We now extend this result by allowing for noisy observations.
 
\subsubsection{Noisy Observations: Continuous 3DVAR}\label{continuous3dvarsection}
Recall that the continuous time limit of 3DVAR is given by \eqref{c3dv}
where the observed data $z$, the integral of $y$, satisfies the SDE \eqref{c3dv1}.
We study this filter in the case where $H=P$ and under small observation noise $\Gamma_0=\epsilon^2I.$ The 3DVAR model covariance is then taken to be of the size of the observation noise.  We choose $C=\sigma^2 I$, where $\sigma^2=\sigma^2(\epsilon)=\eta^{-1}\epsilon^2,$ for some $\eta>0.$ 
Then equations \eqref{c3dv} and \eqref{c3dv1} can be rewritten as
\begin{equation}\label{eq:m}
\frac{\textrm{d}m}{\textrm{d}t}=\mathcal{F}(m)+\frac{1}{\eta}\Big(\frac{\textrm{d}z}{\textrm{d}t}-Pm\Big)
\end{equation}
where
\begin{equation}\label{eq:z}
\frac{\textrm{d}z}{\textrm{d}t}=Pv+\epsilon P\frac{\textrm{d}w}{\textrm{d}t},
\end{equation}
and $w$ is a unit Wiener process.
Note that the parameter $\epsilon$ represents both the size of the 3DVAR observation covariance and the size of the noise in the observations.

The reader will notice that the continuous time synchronization filter
is obtained from this continuous time 3DVAR filter if $\epsilon$ is
set to zero and if the (singular) limit $\eta \to 0$ is taken. 
The next theorem shows that the approximate solution $m$ converges to a neighbourhood of the 
true solution $v$ where the size of the neighbourhood depends upon $\epsilon.$ Similarly as in \cite{lawshuklastuart} and \cite{DLSZ12}  it is required that $\eta,$ the ratio between the size of observation and model covariances, is sufficiently small.
The next theorem is thus a natural generalization of Theorem \ref{t:54}
to incorporate noisy data.

\vspace{0.1in}

\begin{thm}
\label{t:55}
Let $(m,z)$ solve the equations \eqref{eq:m}, \eqref{eq:z} and let $v$ solve 
the equation \eqref{eq:en} with the initial data $v(0) \in {\mathcal B}$,
the absorbing ball of Proposition \ref{absorbing ball}, so that $\sup_{t \ge 0}|v(t)|^2
\le K.$ Then for the constant $c$ as given in the Property \ref{p:2.2}, given 
$\eta <\frac{4}{c^2K}$ we obtain
\begin{equation}\label{45}
\mathbb{E}{|m(t)-v(t)|}^2\le e^{-\lambda t}{|m(0)-v(0)|}^2+\frac{2J\epsilon^2}{3\lambda\eta^2}(1-e^{-\lambda t}),
\end{equation}
where $\lambda$ is defined by 
\begin{equation}\label{56}
\lambda=2\left(1-\frac{c^2 \eta K}{4}\right).
\end{equation}
Thus
$${\rm limsup}_{t \to \infty} \mathbb{E}{|m(t)-v(t)|}^2\le
a\epsilon^2,$$
where $a=\frac{2J}{3\lambda \eta^2}$ does not depend on the strength of the observation noise, $\epsilon.$

\end{thm}

\subsection{Discrete Assimilation}
\label{sec:D}
We now turn to discrete data assimilation. Recall that filters
in discrete time can be split 
into two steps: forecast and analysis.  In this section we establish conditions under which the corrections made at the analysis steps overcome the divergence inherent due to
nonlinear instabilities of the model in the forecast stage. 
As in the previous section we study first the 
case of noiseless data, generalizing the work of \cite{hayden2011discrete} from the Navier-Stokes
and Lorenz '63 models to include the Lorenz '96 model, and then study the
case of 3DVAR, generalizing the work in \cite{brett2012accuracy,lawshuklastuart}, which concerns
the Navier-Stokes and Lorenz '63 models respectively, to the Lorenz '96 model.

\subsubsection{Noiseless Observations}
\label{ssec:D1}

Let $h>0$, and set $t_k:=kh,\,k\ge 0.$ For any function $g:\bbR^+\to\bbR^J,$ continuous in $[t_{k-1},t_k),$ we denote  $g(t_k^-):=\lim_{t\uparrow t_k}g(t).$  
Let $v$ be a solution of equation \eqref{eq:en} with $v(0)$ in
the absorbing forward-invariant ball $\mathcal{B}$. 
The discrete time synchronization filter $m$ of \cite{hayden2011discrete} 
may be expressed as follows: 
\begin{subequations}
\label{8}
\begin{align}
&\frac{\textrm{d}m}{\textrm{d}t}+m+B(m,m)=f, \quad t\in(t_k,t_{k+1}),\\
&m(t_k)=Pv(t_k)+Qm(t_k^-).
\end{align}
\end{subequations}
Thus the filter consists of solving the underlying dynamical model,
by resetting the filter to take the value $Pv(t)$ in the subspace
$P\mathbb{R}^J$ at every time $t=t_k.$ The following theorem shows 
that the filter $m$ converges to the true signal $v.$

\vspace{0.1in}  
\begin{thm}\label{t:56}
Let $v$ be a solution of the equation \eqref{eq:en} with $v(0)\in \mathcal{B}.$ Then there exists $h^*>0$ such that for any $h\in (0,h^*]$ the approximating solution $m$ given by \eqref{8} converges to $v$ as $t\rightarrow \infty$.
\end{thm}

\subsubsection{Noisy Observations: Discrete 3DVAR}\label{noisydiscrete3DVAR}

Now we consider the situation where the data is noisy and $H_k=P.$ We
employ the 3DVAR filter which results from the minimization principle 
\eqref{eq:min} in the case where $\widehat{C}_{k+1}={\sigma}^2I$ and $\Gamma=\epsilon^2I.$ Recall the true signal is determined by the equation \eqref{eq:ds} and the observed data by the equation \eqref{eq:obs}, now written in terms of the true signal $v_k=v(t_k)$ solving the equation \eqref{0} with $v_0\in  \mathcal{B}$. Thus
\begin{align}
{v_{k+1}}& =  {\Psi(v_k),\quad v_0 \in \mathcal{B},}\notag\\
{y_{k+1}}& =  {Pv_{k+1}+\nu_{k+1}}.\notag
\end{align}
If we define $\eta:= \frac{{\epsilon}^2}{{\sigma}^2}$ then the  3DVAR
filter can be written as 
\begin{equation*}
m_{k+1}= \Big( \frac{\eta}{1+\eta}P+Q \Big)\Psi(m_k)+ \frac{1}{1+\eta}y_{k+1},
\end{equation*}
after noting that $Py_{k+1}=y_{k+1}$ because $P$ is a projection and $\nu_{k+1}$ is
assumed to lie in the image of $P$.
In fact the data has the following form:
\begin{align*}
y_{k+1} &= Pv_{k+1}+P\nu_{k+1}\notag \\
&= P\Psi (v_k)+\nu_{k+1}.
\end{align*}

Combining the two equations gives
\begin{equation}\label{eq:me}
m_{k+1}= \Big( \frac{\eta}{1+\eta}P+Q \Big)\Psi(m_k)+ \frac{1}{1+\eta}\Bigl(P\Psi (v_k)+\nu_{k+1}\Bigr).
\end{equation}
We can write the equation for the true solution $v_k$, given by \eqref{eq:ds},
in the following form:
\begin{equation}\label{eq:ve}
v_{k+1}= \Big( \frac{\eta}{1+\eta}P+Q \Big)\Psi(v_k)+ \frac{1}{1+\eta}P\Psi (v_k).
\end{equation}
Note that $v_k=v(t_k)$ where $v(\cdot)$ solves \eqref{eq:en}.
We are interested in comparing the output of the
filter, $m_k,$ with the true signal $v_k.$
Notice that if the noise $\nu_k$ is set to zero and if the limit $\eta \to 0$
is taken then the filter becomes
$$m_{k+1}=P\Psi (v_k)+Q\Psi(m_k)$$
which is precisely the discrete time synchronization filter.
Theorem \ref{thm:dis} below will reflect this observation,
constituting a noisy variation on Theorem \ref{t:56}.

We will assume that the $\nu_k$ are independent random variables that satisfy the bound $|\nu_k|\le \epsilon$,
thereby linking the scale of the covariance $\Gamma$ employed in 3DVAR to
the size of the noise.
We let $\|\cdot\|$ be the norm defined by $\|z\|:=|z|+|Pz|,$ $z\in \mathbb{R}^J.$
\bigskip

\begin{thm}\label{thm:dis}
Let $v$ be the solution of the equation (\ref{eq:en}) with $v(0)\in \mathcal{B}.$ Assume that $\{\nu_k\}_{k\ge 1}$ is a sequence of independent bounded random variables such that, for every $k,$ $|\nu_k|\le \epsilon.$
Then there are choices (detailed in the proof in the appendix) of assimilation step $h>0$ and parameter $\eta>0$ sufficiently small such that, for some  $\alpha\in(0,1)$ and provided that the noise $\epsilon>0$ is small enough,  the error satisfies
\begin{equation}
\label{eq:bnd}
\|m_{k+1}-v_{k+1}\|\le \alpha{\|m_k-v_k\|}+2{\epsilon}.
\end{equation}
Thus, there is $a>0$ such that
$$\limsup_{k\to\infty} \|m_k-v_k\|\le a\epsilon.$$ 

\end{thm}

\section{Adaptive Observation Operator}\label{sec:N}

The theory in the previous section demonstrates that
accurate filtering of chaotic models is driven by observing enough of the 
dynamics to control the exponential separation of trajectories in the
dynamics. However the fixed observation operator $P$ that we 
analyze requires observation of $2/3$ of the system state vector.
Even if the observation operator is fixed our numerical results
will show that observation of this proportion of the state is not 
necessary to obtain accurate filtering. Furthermore, by adapting the
observations to the dynamics, we will be able to obtain the same quality of
reconstruction with even fewer observations. In this section
we will demonstrate
these ideas in the context of noisy discrete time filtering, and
with reference to the Lorenz '96 model. 

The variational equation for the dynamical system \eqref{mthclf} is given by 
\begin{equation}\label{eq:vs}
\frac{\textrm{d}}{\textrm{d}t} D\Psi(u,t)=D{\mathcal F}\bigl(\Psi(u,t)\bigr)\cdot D\Psi(u,t);\quad \quad D\Psi(u,0)=I_{J\times J},
\end{equation}
using the chain rule.
The solution of the variational equation gives the derivative matrix of the solution operator $\Psi$, which in turn characterizes the behaviour of $\Psi$ with respect to small variations in the initial value $u$. 
Let $L_{k+1}:=L(t_{k+1})$ be the solution of the 
variational equation \eqref{eq:vs} over the assimilation window $(t_{k},t_{k+1})$, initialized at $I_{J\times J}$, given as 
\begin{equation}
\label{eq:leq}
\frac{\textrm{d}L}{\textrm{d}t}=D{\mathcal F}\bigl(\Psi(m_{k},t-t_{k})\bigr)L, 
\quad t \in (t_{k},t_{k+1}); \quad\quad L(t_{k})=I_{J\times J}.
\end{equation}
Let $\{\lambda_k^j,\psi_k^j\}_{j=1}^J$ denote eigenvalue/eigenvector pairs
of the matrix $L_{k+1}^TL_{k+1},$ where the eigenvalues (which are, of course, real)
are ordered to be non-decreasing, and the eigenvectors are orthonormalized
with respect to the Euclidean inner-product $\langle \cdot, \cdot \rangle.$
We define the adaptive observation operator 
$H_k$ to be 
\begin{equation}
\label{eq:hk}
H_k:=H_0(\psi_k^1,\cdots,\psi_k^J)^T
\end{equation} 
where
\begin{eqnarray}
H_0=\left(
\begin{array}{cc}
0 & 0\\
0 & I_{M \times M}
\end{array}
\right).
\end{eqnarray}
Thus $H_0$ and $H_k$ both have rank $M$.
Defined in this way we see that for any given $v\in \mathbb{R}^J$ the projection $H_k v$ is given by
the vector 
$$\Bigl(0,\cdots, 0, \langle \psi_k^{J-M+1}, v \rangle, \cdots, \langle \psi_k^{J}, v \rangle\Bigr)^T,$$
that is the projection of $v$ onto the $M$ eigenvectors of $L_{k+1}^TL_{k+1}$
with largest modulus.

\begin{remark}
In the following work we consider the leading eigenvalues and corresponding eigenvectors of the matrix $L_k^TL_k$ to track the unstable (positive Lyapunov growth) directions. To leading order in $h$ it is equivalent to consider the matrix $L_kL_k^T$ in the case of frequent observations (small $h$) as can be seen by the following expressions
\begin{eqnarray}
L_k^TL_k&=&(I+hD\mathcal{F}_k)^T(I+hD\mathcal{F}_k)+\mathcal{O}(h^2)\notag\\
&=&I+h(D\mathcal{F}_k^T+D\mathcal{F}_k)+\mathcal{O}(h^2)\notag
\end{eqnarray}
and
\begin{eqnarray}
L_kL_k^T&=&(I+hD\mathcal{F}_k)(I+hD\mathcal{F}_k)^T+\mathcal{O}(h^2)\notag\\
&=&I+h(D\mathcal{F}_k+D\mathcal{F}_k^T)+\mathcal{O}(h^2),\notag
\end{eqnarray}
where $D\mathcal{F}_k= D\mathcal{F}(m_k)$.

{Of course for large intervals $h$, the above does not hold, and the difference between $L_k^TL_k$ and $L_kL_k^T$ may be substantial.
It is however clear that these operators have the same eigenvalues, with the eigenvectors of $L_kL_k^T$ corresponding to $\lambda_k^j$ 
given by $L_k \psi_k^j$ for the corresponding eigenvector $\psi_k^j$ of $L_k^T L_k$.  That is to say, for the linearized deformation map $L_k$, 
the direction $\psi_k^j$ is the pre-deformation principle direction corresponding to the principle strain $\lambda_k^j$ induced by the deformation.
The direction $L_k \psi_k^j$ is the post-deformation principle direction corresponding to the principle strain $\lambda_k^j$.  
The dominant directions chosen in Eq. \eqref{eq:hk} are those directions corresponding to the greatest growth over the interval $(t_k,t_{k+1})$ of infinitesimal perturbations to the predicting trajectory, $\Psi(m_{k-1},h)$ at time $t_k$.  This is only one sensible option.  One could alternatively consider the directions corresponding to the greatest growth over the interval
$(t_{k-1},t_{k})$, or over the whole interval $(t_{k-1},t_{k+1})$.
Investigation of these alternatives is beyond the scope of this work and is therefore deferred to later investigation.
}

\end{remark}

We make a small shift of notation and now consider
the observation operator $H_k$ as a linear mapping from $\bbR^J$ into $\bbR^M$,
rather than as a linear operator from $\bbR^J$ into itself, with rank $M$;
the latter perspective was advantageous for the presentation of the analysis,
but differs from the former which is sometimes computationally advantageous and 
more widely used for the description of algorithms. Recall the minimization principle \eqref{eq:min}, noting
that now the first norm is in $\bbR^J$ and the second in $\bbR^M$.

\subsection{3DVAR}
\label{ssec:3}

Here we consider the minimization principle \eqref{eq:min} with the choice
{$\widehat{C}_{k+1}=C_0 \in \bbR^{J \times J}$, a strictly positive-definite matrix,
for all $k$. Assuming that $\Gamma \in \bbR^{M \times M}$ is also
strictly positive-definite, the filter may be written as\as{
\begin{subequations}
\label{equations}
\begin{align}
  \label{eq:f}
  m_{k+1} &= \Psi(m_{k})+G_{k+1}\Big(y_{k+1}-H_{k+1}\Psi(m_{k})\Big)\\
  \label{eq:g}
  G_{k+1} &= C_0H_{k+1}^T(H_{k+1}C_0H_{k+1}^T+\Gamma)^{-1}.
\end{align}
\end{subequations}
}
As well as using the choice of $H_k$ defined in \eqref{eq:hk},
we also employ the fixed observation operator where $H_k=H$, including the
choice $H=P$ given by \eqref{P}. In the last case $J=3J'$, 
$M=2J'$ and $P$ is realized as a $2J'\times 3J'$ matrix.

We make the 
choices $C_0=\sigma^2 I_{J \times J},\, \Gamma=\epsilon^2 I_{M \times M}$ 
and define $\eta=\epsilon^2/\sigma^2.$ Throughout our experiments we take $h=0.1$, 
$\epsilon^2=0.01$ and fix the parameter $\eta= 0.01$ (i.e. $\sigma=1$). 
We use the Lorenz '96 model \eqref{1} to define $\Psi$, with the parameter choices
$F=8$ and $J=60$. The system then has $19$ positive Lyapunov exponents which we calculate by the methods described in \cite{PhysRevA.14.2338}. 
The 
observational noise is i.i.d Gaussian with respect to time index $k$,
with distribution $\nu_1 \sim N(0,\epsilon^2)$.

Throughout the following we show (approximation) to the expected
value, with respect to noise realizations around a single fixed 
true signal solving \eqref{eq:en},
of the error between the filter and the signal underlying
the data, in the Euclidean norm, as a function of time. We also
quote numbers which are found by time-averaging this quantity.
The expectation is approximated by a Monte Carlo method
in which $I$ realizations of the noise in the data are created, leading
to filters $m_k^{(i)}$, with $k$ denoting time and $i$ denoting
realization. Thus we have, for $t_k=kh$,
$${\rm RMSE}(t_k)=\frac{1}{I}\underset{i=1}{\overset{I}{\sum}} \sqrt{\frac{{\Vert m^{(i)}_k-v_k\Vert}^2}{J}}.$$
This quantity is graphed, as a function of $k$, in what follows.
Notice that similar results are obtained if only one realization is used
($I=1$) but they are more noisy and hence the trends underlying
them are not so clear. We take $I=10^4$ throughout the reported
numerical results. When we state a number for the ${\rm RMSE}$
this will be found by time-averaging after ignoring the initial transients ($t_k<40$):

$$ {\rm RMSE}={\underset{t_k > 40}{\rm mean} } \{{\rm RMSE}(t_k)\}. $$
In what follows we will simply refer to \RMSE; from the
context it will be clear whether we are talking about the
function of time, ${\rm RMSE}(t_k)$, or the time-averaged number
${\rm RMSE}$. 

Figures 
\ref{fig:fixed40} and \ref{fig:two} exhibit, for fixed observation 3DVAR 
and adaptive observation 3DVAR, the \RMSE as a function of time. 
The Figure \ref{fig:fixed40} shows the \RMSE for fixed observation 
operator where the observed space is of dimension $60$ (complete
observations), $40$ (observation operator defined as in the equation \eqref{P}), \as{$36$}
and $24$ respectively.
For \as{values $M=60$, $40$ and $36$} the error decreases rapidly and the approximate solution 
converges to a neighbourhood of the true solution where the size of the neighbourhood 
depends upon the variance of the observational noise. \as{For the cases $M=60$ and $M=40$ we use the identity operator $I_{J\times J}$ and the projection operator $P$ as defined in the equation \eqref{P} as the observation operators respectively. The observation operator for the case $M=36$ can be given as 
\begin{equation} P_{36}=\left( \begin{array}{c}\label{P1}
e_1,\, e_2,\,  0,\,  e_4,\,  0,\,  e_6, e_7,\, 0,\,  e_9,\,  0, \, e_{11},\,  e_{12},\,  0,\,  e_{14}, \ldots 
\end{array}\right)_{J\times J}\end{equation} where we observe $3$ out of $5$ directions periodically. The \RMSE, averaged over the 
trajectory, after ignoring the initial transients, is $1.30\times 10^{-2}$ when $M=60$,  $1.14\times 10^{-2}$ when $M=40$ and \as{$1.90\times 10^{-2}$ when $M=36;$} note that this is on the scale of the observational noise. 
{The rate of convergence of the approximate solution to the true solution in the case of partial observations is lower than the rate of convergence when full observations are used. However, despite this, the RMSE itself is lower in the case when $M=40$ than in the case of full observations. We conjecture that this is because there is, overall, less noise injected into the system when $M=40$ in comparison to the case when all directions are observed.
} \as{The convergence of the approximate solution to the true solution for the case when $M=36$ shows that the value $M=40$, for which theoretical results have been presented in section \ref{sec:C}, is not required for small error (${\cal O}(\epsilon)$) consistently over the trajectory.}}
We also consider the case when $24=40\%$ of the modes are observed \as{using
the following observation operator: 
\begin{equation} P_{24}=\left( \begin{array}{c}\label{P2}
e_1,\, 0,\,  0,\,  e_4,\,  0,\,  0, e_7,\, 0,\,  0,\,  e_{10}, \, e_{11},\, 0,\,  0,\,  e_{14}, \ldots 
\end{array}\right)_{J\times J}.\end{equation}
Thus we observe $4$ out of $10$ directions periodically;} this \as{structure}
is motivated by the work reported in \cite{hdia13,kostuk12} where it was demonstrated that
observing $40\%$ of the modes, with the observation directions chosen carefully and with observations sufficiently frequent in time, is 
sufficient for the approximate solution to converge to the true underlying solution. The Figure \ref{fig:fixed40} shows that, in our observational set-up, observing $24$
of the modes only allows marginally successful reconstruction of
the signal, asymptotically in time;
the \RMSE makes regular large excursions 
and the
time-averaged RMSE over the trajectory is $(5.73\times 10^{-2})$, which is 
an order of magnitude larger than for $36$, $40$ or $60$ observations.


\begin{figure}[H]

        \centering
\includegraphics[width=17cm, height=8cm]{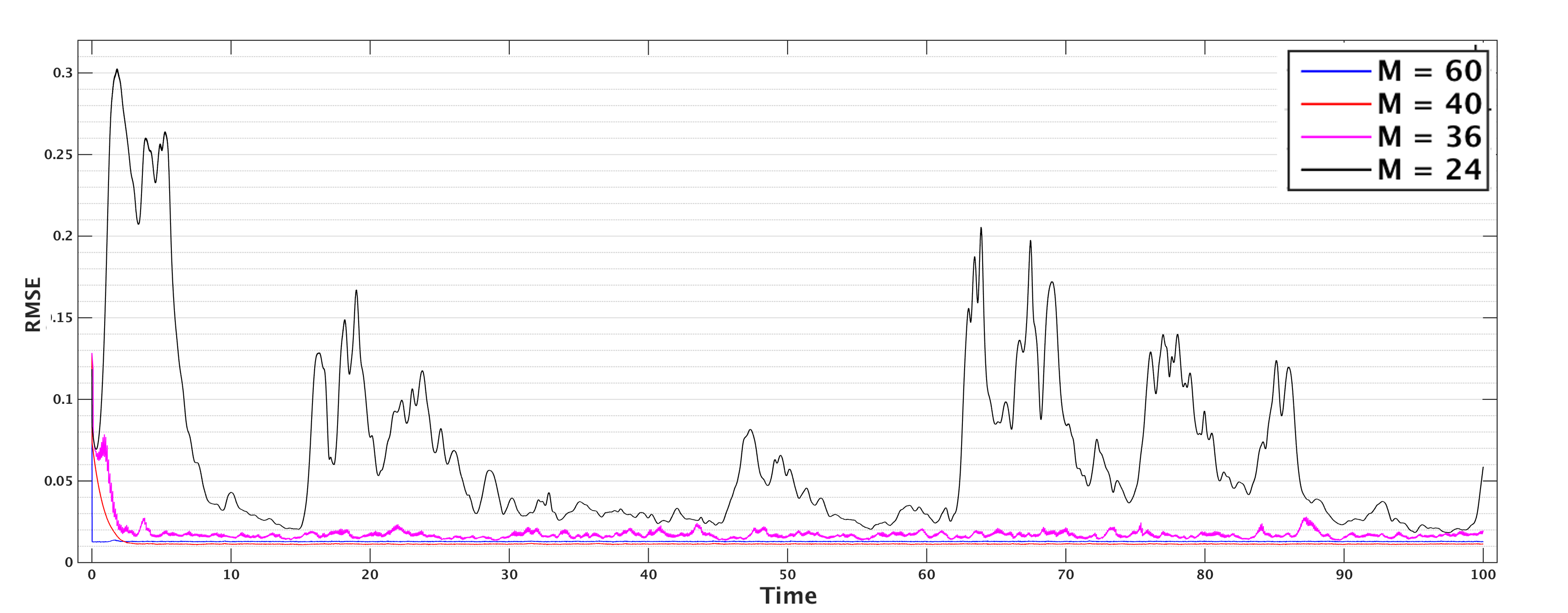}
                \caption{Fixed Observation Operator 3DVAR. Comparison with the case  when $M=24$. RMSE  value averaged over the trajectory  for $M=24$ is $5.73\times 10^{-2}$. }
                \label{fig:fixed40}
   
 \end{figure}     

Figure \ref{fig:two} shows the \RMSE for adaptive observation 3DVAR. 
In this case we notice that the error is consistently
small, uniformly in time, with just $9$ or more modes observed. 
When $M=9$ ($15\%$ observed modes) the \RMSE averaged over the trajectory is $1.35\times 10^{-2}$ which again is of the order of the observational noise variance.
For $M \ge 9$ the error is similar -- see Figure \ref{fig:tone}.
On the other hand, for smaller values of $M$ the error is not controlled as
shown in Figure \ref{fig:ttwo} where the \RMSE for $M=7$ is compared
with that for $M=9;$ for $M=7$ it is an order of magnitude larger than
for $M=9$. It is noteworthy that the number of observations necessary and
sufficient for accurate reconstruction is approximately half the number 
of positive Lyapunov exponents.

\begin{figure}
	
		        \begin{subfigure}[b]{\textwidth}
                \includegraphics[width=\textwidth]{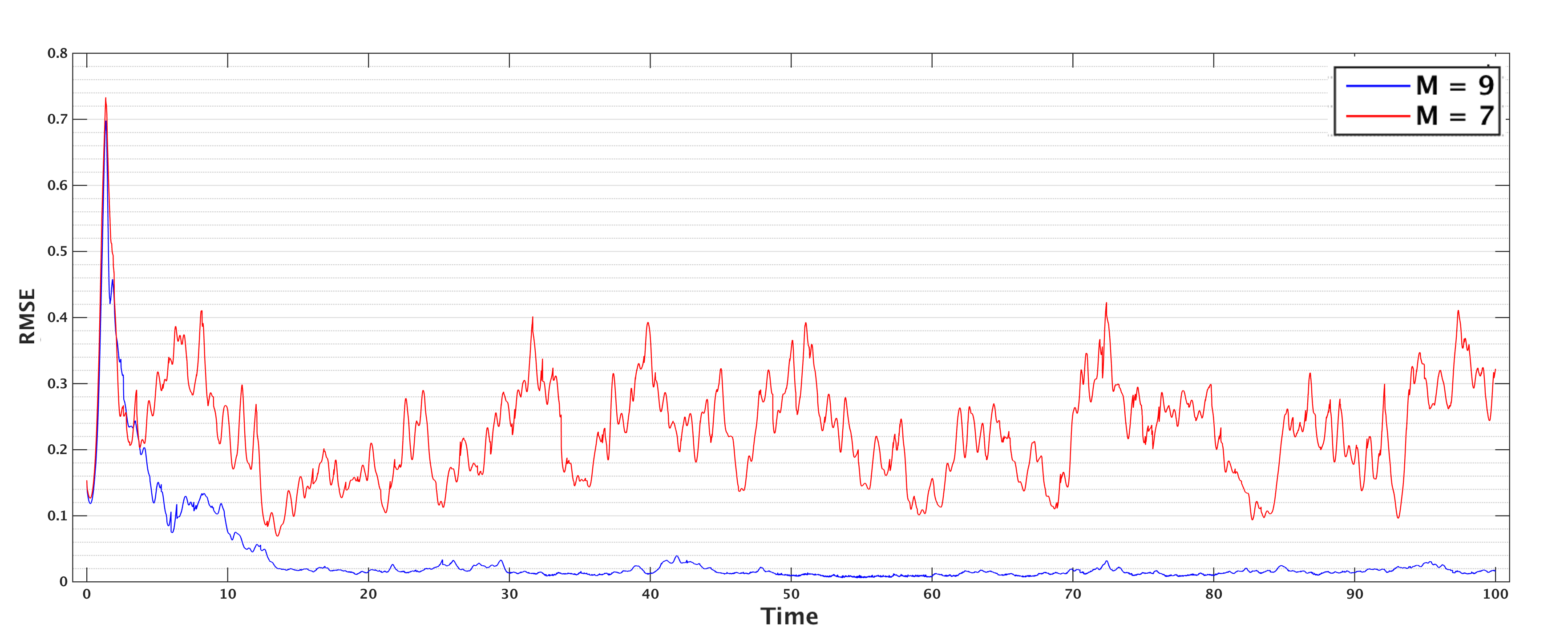}
                \caption{Comparison of RMSE between M = 7 and M = 9. RMSE  values averaged over trajectory are $2.25\times 10^{-1}$, $1.35\times 10^{-2}$ respectively.}
                \label{fig:ttwo}
                \end{subfigure}
        \begin{subfigure}[b]{\textwidth}
                \includegraphics[scale=0.80]{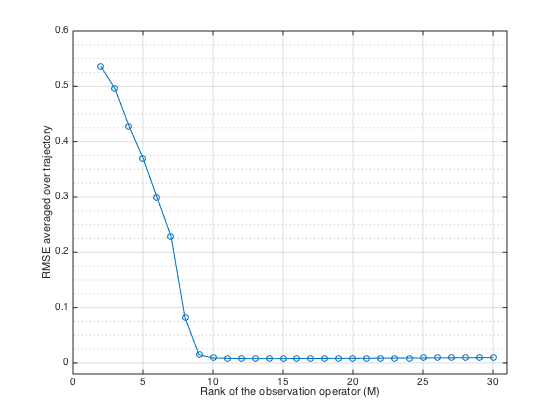}
                \caption{Averaged RMSE for different choices of $M$}
                \label{fig:tone}
        \end{subfigure}        ~ 
                ~ 
           
        \caption{Adaptive Observation 3DVAR}\label{fig:two}
\end{figure}   
\subsection{Extended Kalman Filter}
\label{ssec:EX}

In the Extended Kalman Filter (ExKF) the approximate solution evolves according to 
the minimization principle \eqref{eq:min} with $C_k$ chosen as a covariance matrix
evolving in the forecast step according to the linearized dynamics, and in
the assimilation stage updated according to Bayes' rule based on a Gaussian
observational error covariance. This gives the method
\begin{eqnarray*}
m_{k+1} &=& \Psi(m_{k})+G_{k+1}\Big(y_{k+1}-H_{k+1}\Psi({m}_{k})\Big),\\
\widehat{C}_{k+1} &=& D\Psi(m_k) C_{k}D\Psi(m_k)^T,\\
C_{k+1}&=&(I_{J\times J}-G_{k+1}H_{k+1})\widehat{C}_{k+1},\\
G_{k+1} &=& \widehat{C}_{k+1}H_{k+1}^T(H_{k+1}\widehat{C}_{k+1}H_{k+1}^T+\Gamma)^{-1}.
\end{eqnarray*}

We first consider the ExKF scheme with a fixed observation operator $H_{k}=H.$
We make two choices for $H$: the full rank identity operator and a partial 
observation operator given by \eqref{P2} so that $40\%$ of the modes are 
observed. For the first case the filtering scheme is the standard ExKF with all the modes being observed. The approximate solution converges to the true solution and the error decreases rapidly as can be seen in the Figure \ref{fig:ekf60}. The \RMSE is $9.49\times 10^{-4}$ 
which is an order of magnitude smaller than the analogous error for the
3DVAR algorithm when fully observed which is, recall, $1.30\times 10^{-2}$.
For the partial observations case with $M=24$ we see that again the approximate solution converges to the true underlying solution as shown in the Figure \ref{fig:ekf24}. 
Furthermore the solution given by the ExKF with $M=24$ is far more robust than
for 3DVAR with this number of observations. The \RMSE is 
also lower 
for ExKF $(2.68\times 10^{-3})$ when compared with 
the 3DVAR scheme $(5.73\times 10^{-2})$. 

\begin{figure}
        \begin{subfigure}[b]{\textwidth}
                \includegraphics[width=\textwidth]{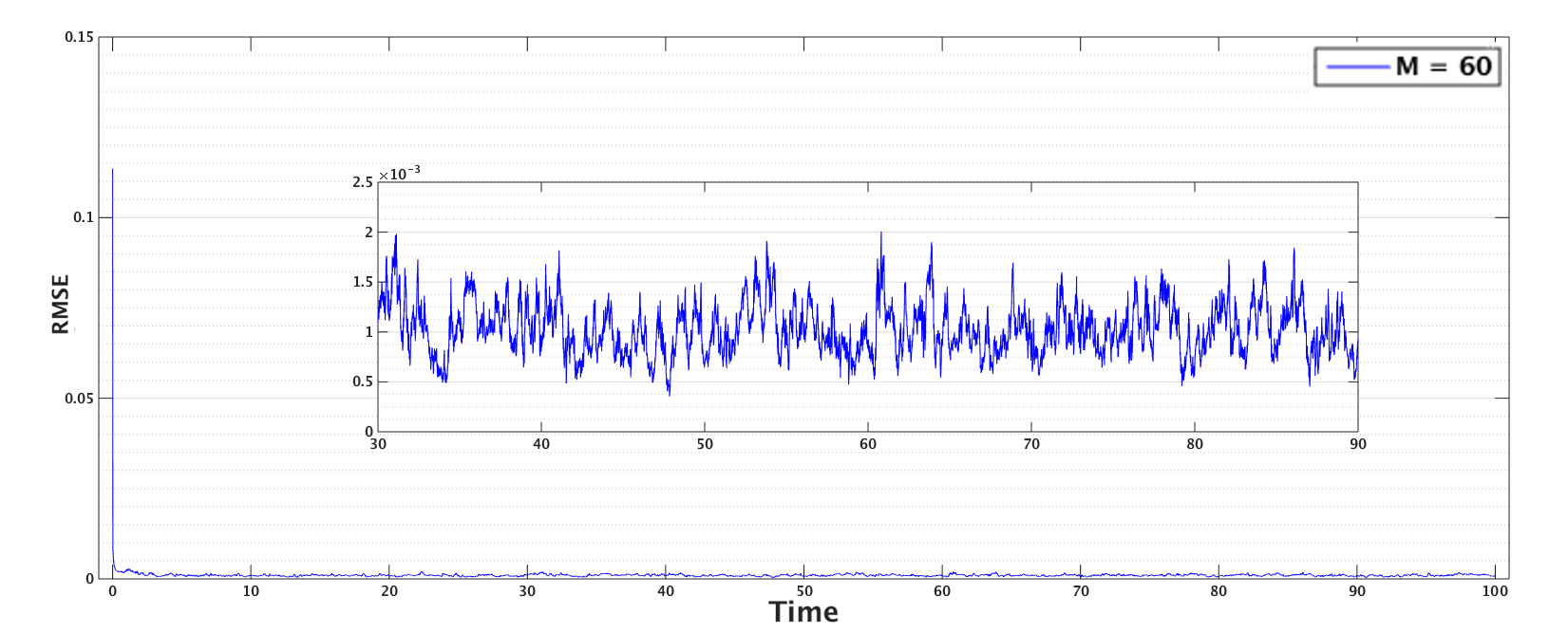}
                \caption{Percentage of components observed = $100\%$. RMSE  value averaged over trajectory $9.49\times 10^{-4}. $}
                \label{fig:ekf60}
        \end{subfigure}%
        ~           \\
        \begin{subfigure}[b]{\textwidth}
                \includegraphics[width=\textwidth]{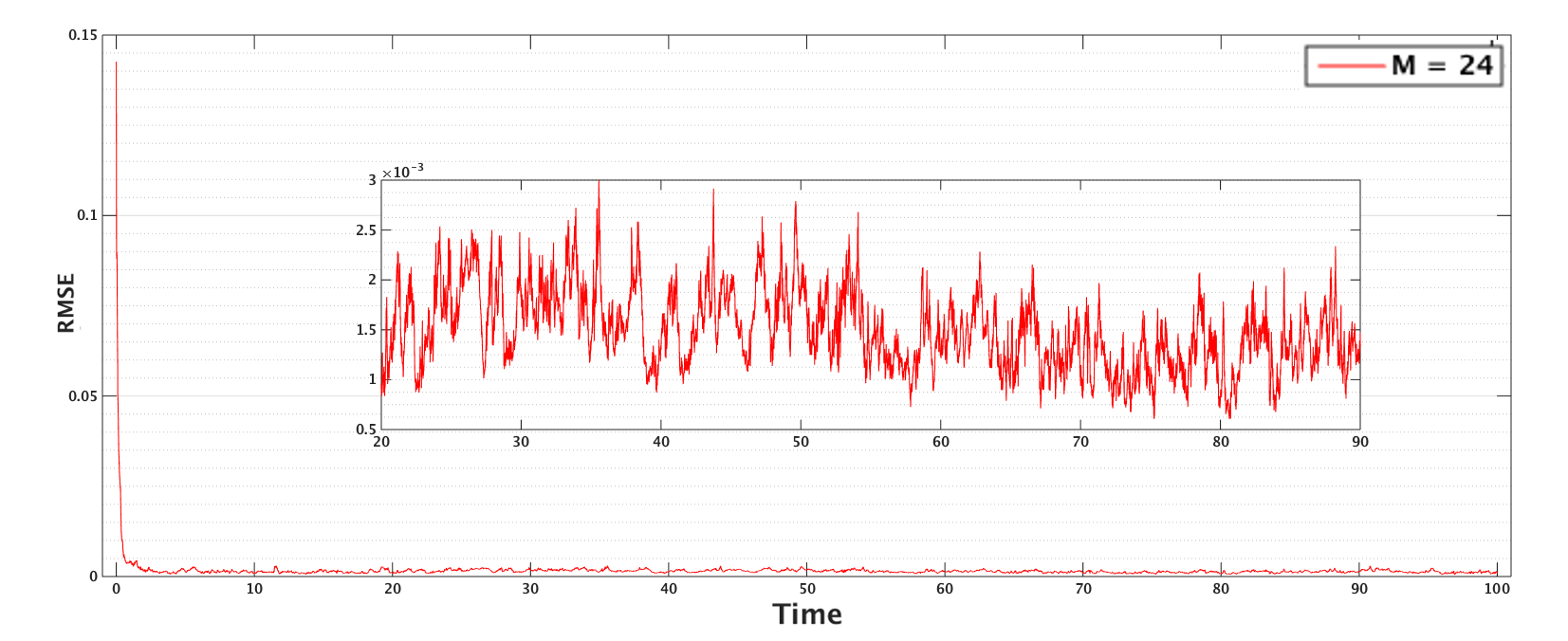}
                \caption{Percentage of components observed = $40\%$. RMSE  value averaged over trajectory $1.39\times 10^{-3}. $}
                \label{fig:ekf24}
        \end{subfigure}
        ~            
        \caption{Fixed Observation ExKF. The zoomed in figures shows the variability in RMSE between time $t=20$ and $t=90$.}\label{fig:ekffo}
\end{figure}   

We now turn to adaptive observation within the context of the ExKF.
The Figure \ref{fig:ekfad} shows that it is possible to obtain an \RMSE
which is of the order of the observational error, and is robust
over long time intervals, {using only a $6$ dimensional observation space},
improving marginally on the 3DVAR situation where $9$ dimensions were required to attain a similar level of accuracy.

\begin{figure}

			 \begin{subfigure}[b]{\textwidth}
                \includegraphics[width=\textwidth]{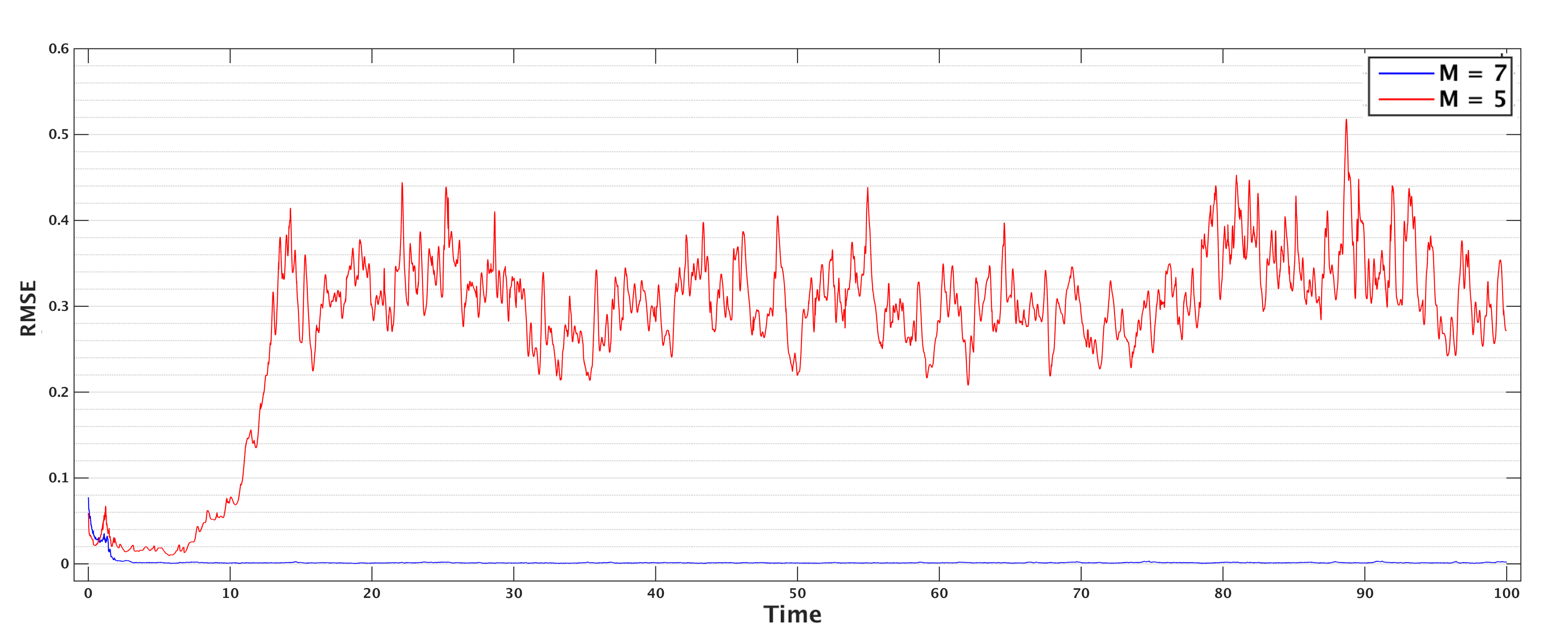}
                \caption{Comparison of RMSE between {M = 5 and M = 7}. RMSE  values averaged over trajectory are $2.84\times 10^{-1}$, $1.31\times 10^{-3}$ respectively.}
                \label{fig:ekf7}
        \end{subfigure}
        \begin{subfigure}[b]{\textwidth}
                \includegraphics[scale=0.70]{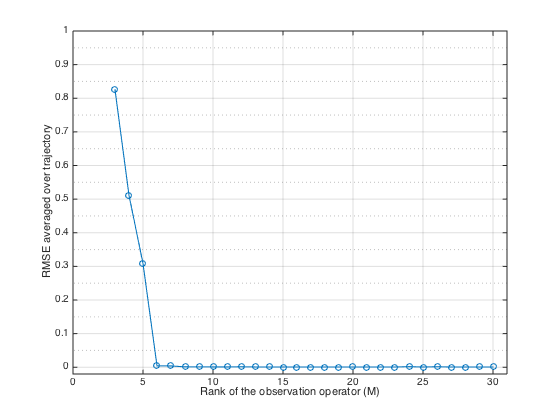}
                \caption{Averaged RMSE 
                for different choices of $M$.}
                \label{fig:ekfm}
        \end{subfigure}%
        ~           \\
       
        ~            
        \caption{Adaptive Observation ExKF}\label{fig:ekfad}
\end{figure}

\rev{The AUS scheme, proposed by Trevisan and co-workers \cite{trevisan04, trevisan_npg2011}, is an ExKF
method which operates by confining the analysis update to a subspace designed to capture the instabilities in the dynamics. 
This subspace is typically chosen as the span of the $M$ 
largest growth directions, where $M$ is  the precomputed number of non-negative Lyapunov exponents.
To estimate the unstable subspace one starts with $M$ orthogonal perturbation vectors and propagates them forward under the linearized dynamics in the forecast step to obtain a forecast  covariance matrix $(\hat{C_k}).$ The perturbation vectors for the next assimilation cycle are 
provided by 
the square root of the covariance matrix $(C_k)$ which can be computed via a suitable $M\times M$
transformation as shown in equations (11)-(15) of \cite{trevisan_npg2011}. 
Under the assumption that the observational noise is sufficiently small that the truth of the exact model 
is close to the estimated mean and the discontinuity of the update is not too significant, it can be argued
that the unstable subspace generated by the dominant Lyapunov vectors is preserved through the assimilation cycle.  This has been illustrated numerically in \cite{trevisan_npg2011} and references therein.  That work also observes the phenomenon of reduced error in the AUS scheme as compared to the full assimilation, due to corruption by observational noise in stable directions in the latter case. }
Asymptotically this method with $H=I_{J \times J}$ behaves similarly
to the adaptive ExKF with {observation operator  of rank} $M$. 
To understand the intuition behind the AUS method we plot in Figure \ref{fig:ekfrank}
the rank (computed by truncation to zero of eigenvalues below a threshold)
of the covariance matrix $C_k$ from standard ExKF based on observing
$60$ and $24$ modes. Notice that in both cases the rank approaches a value of
$19$ or $20$ and that $19$ is the number of non-negative Lyapunov exponents.
This means that the covariance is effectively zero in $40$ of the observed dimensions
and that, as a consequence of the minimization principle \eqref{eq:min}, data
will be ignored in the $40$ dimensions where the covariance is negligible.
It is hence natural to simply confine the update step to the subspace of
dimension $19$ given by the number of positive Lyapunov exponents, right from
the outset. This is exactly what AUS does by reducing the rank of the error covariance matrix $C_k$. 
Numerical results are given in
Figure \ref{fig:aus} which shows the \RMSE 
over the trajectory for the ExKF-AUS assimilation scheme {versus time for the observation operator $H=I_{J \times J}$}. After 
initial transients the error is mostly of the numerical order of the observational noise. Occasional jumps outside this error bound are observed but the approximate solution converges to the true solution each time. The \RMSE for ExKF-AUS is $1.49\times 10^{-2}$. However, if the rank of the error covariance matrix $C_0$ in AUS
is chosen to be less than the number of unstable modes for the underlying system, 
then the  approximate solution does not converge to the true solution. 
\begin{figure}

  \begin{subfigure}[b]{\textwidth}
                \includegraphics[scale=0.60]{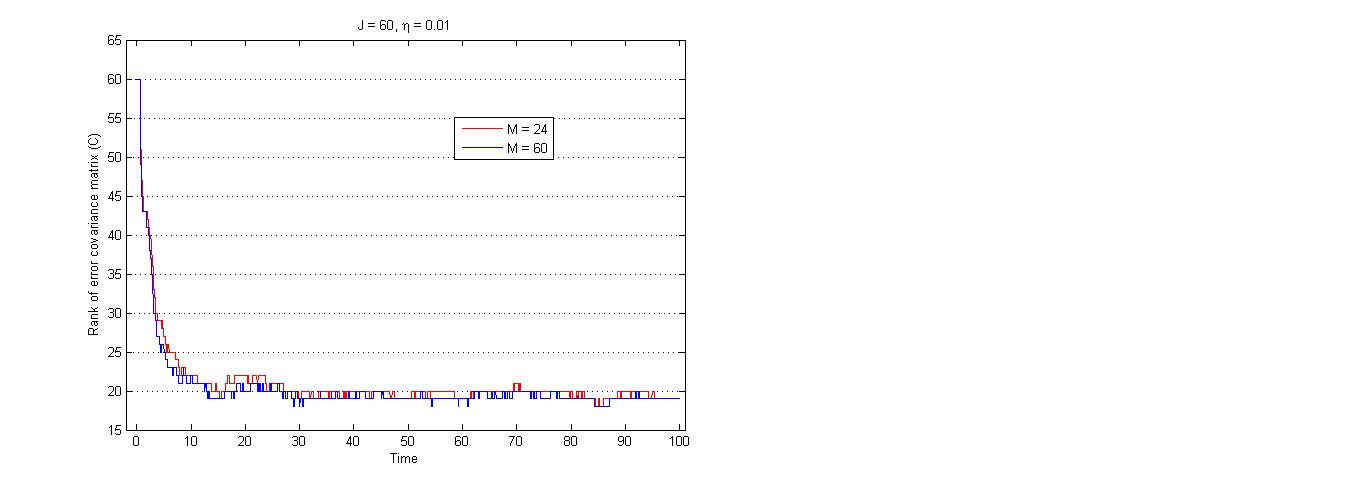}
                \caption{Standard ExKF with $60$ and $24$ observed modes. The rank of the error covariance matrix $C_k$ decays to (approximately) the number of unstable Lyapunov modes in the underlying system, namely $19$. }
                \label{fig:ekfrank}
        \end{subfigure}%
        ~           \\

        \begin{subfigure}[b]{\textwidth}
                \includegraphics[width=\textwidth]{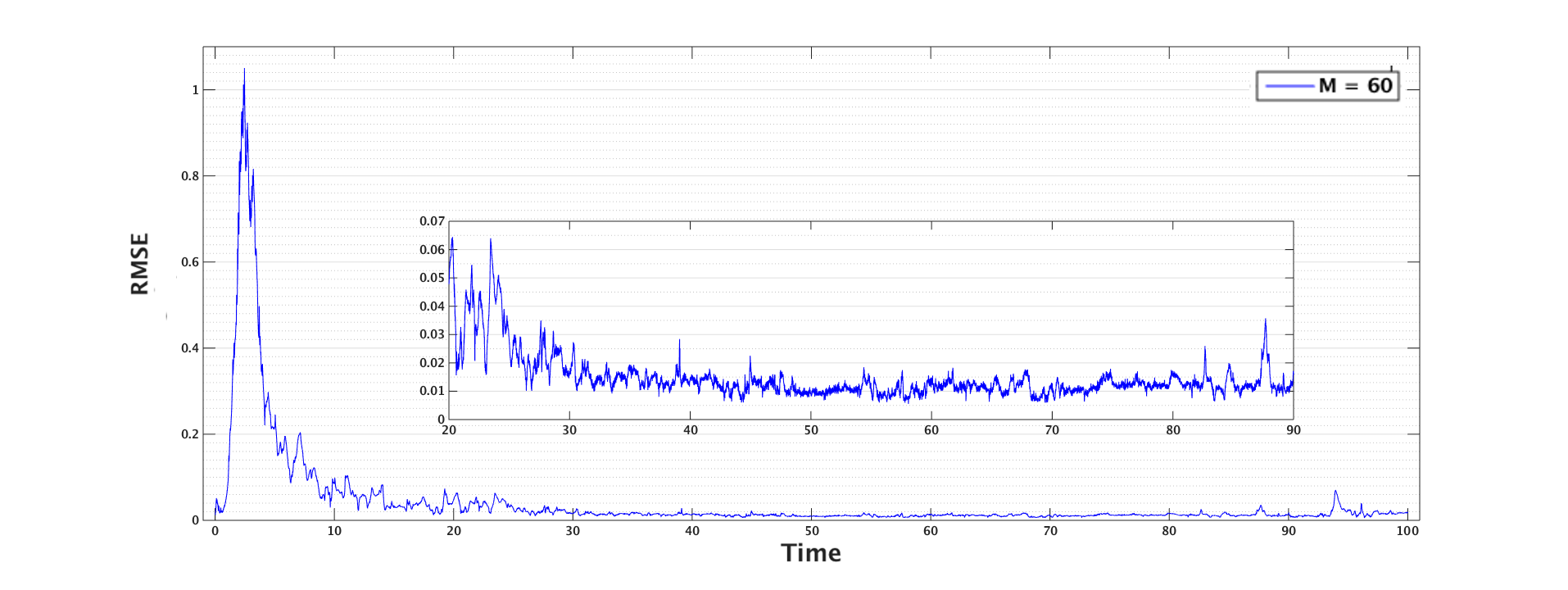}
                \caption{ RMSE  value averaged over trajectory: $1.49\times 10^{-2}. $ The zoomed in figures shows the variability in RMSE between time $t=20$ and $t=90$.  The rank of observation operator is chosen $M=60$.}
                \label{fig:aus}
        \end{subfigure}
        ~         \caption{Rank of error covariance and ExKF-Assimilation in Unstable Space}\label{fig:ekfaus}
\end{figure} 

\section{Conclusions}
\label{sec:CO}

In this paper we have studied the long-time behaviour of filters for
partially observed dissipative dynamical systems, using the Lorenz '96
model as a canonical example. We have highlighted the connection to
synchronization in dynamical systems, and shown that this synchronization 
theory, which applies to noise-free data, is robust to the addition of
noise, in both the continuous and discrete time settings. 
In so doing we are studying the 3DVAR algorithm. In the context of the Lorenz '96 model we have identified a fixed observation
operator, based on observing 2/3 of the components of the signal's vector,
which is sufficient to ensure desirable long-time properties of the
filter. However it is to be expected
that, within the context of fixed observation operators, 
considerably fewer observations may be needed to ensure such desirable properties.  
Ideas from nonlinear control theory will be relevant in addressing this issue. We also studied adaptive observation operators, targeted to observe the directions of maximal growth within the local linearized dynamics. We demonstrated that with these adaptive observers, considerably fewer observations are required. We also made a connection between
these adaptive observation operators, and the AUS methodology which
is also based on the local linearized dynamics, but works by
projecting within the model covariance operators of ExKF, whilst
the observation operators themselves are fixed; thus the model covariances
are adapted. Both adaptive observation operators and the AUS methodology
show the potential for considerable computational savings in filtering,
without loss of accuracy.

In conclusion our work highlights the role of ideas from dynamical systems
in the rigorous analysis of filtering schemes and, through computational
studies, shows the gap between theory and practice, demonstrating
the need for further theoretical developments. {We emphasize that
the adaptive observation operator methods may not be implementable
in practice on the high dimensional systems arising in, for example,
meteorological applications.} However, they provide conceptual insights into the
development of improved algorithms and it is hence important to understand
their properties.

\noindent {\bf Acknowledgements.  }{
AbS and DSA are supported by the EPSRC-MASDOC graduate training scheme.
AMS is supported by EPSRC, ERC and ONR.
KJHL is supported by King Abdullah University of Science and Technology, 
and is a member of the KAUST SRI-UQ Center.}

\vspace{0.5in}
\bibliographystyle{plain}
\bibliography{mybibfile}

\section*{Appendix: Proofs}

\begin{proof}[Proof of Properties \ref{p:2.1}]
Properties \ref{itm:a}, \ref{itm:b} and \ref{itm:c} are straightforward and we omit the proofs. 
We start showing \ref{itm:d}. 
For any $u\in\mathbb{R}^J$ set
$$\|u\|_{\infty}=\underset{1\le j\le J}{\max}|u^{(j)}|$$ and recall that  $|u|^2\geq \|u\|_{\infty}^2.$ Then, for $u,\tilde{u}\in \mathbb{R}^J$, 
 and for $1\le j \le J$, we have that
\begin{eqnarray*}
2|B(u, \tilde{u})^{(j)}| &\le & \|u\|_{\infty}(|\tilde{u}^{(j+1)}|+|\tilde{u}^{(j-2)}|)+\|\tilde{u}\|_{\infty}(|u^{(j+1)}|+|u^{(j-2)}|),
\end{eqnarray*}
and so
\begin{eqnarray*}
4|B(u, \tilde{u})|^2 &\le & 2\|u\|^2_{\infty}\overset{J}{\underset{j=1}{\sum}}(|\tilde{u}^{(j+1)}|+|\tilde{u}^{(j-2)}|)^2+2\|\tilde{u}\|^2_{\infty}\overset{J}{\underset{j=1}{\sum}}(|u^{(j+1)}|+|u^{(j-2)}|)^2\notag\\
&\le & 8\|u\|^2_{\infty}|\tilde{u}|^2+8\|\tilde{u}\|^2_{\infty}|u|^2\notag\\
&\le & 16|u|^2|\tilde{u}|^2.
\end{eqnarray*}
Hence
\begin{equation*}
|B(u, \tilde{u})|\le 2|u||\tilde{u}|.
\end{equation*}
For \ref{itm:f} we use rearrangement and periodicity of indices under summation as follows:
\begin{eqnarray*}
2\langle B(u,\tilde{u}),u\rangle &=&\overset{J}{\underset{j=1}{\sum}}\Big(u^{(j)}(u^{(j-1)}\tilde{u}^{(j+1)}+\tilde{u}^{(j-1)}u^{(j+1)}-\tilde{u}^{(j-1)}u^{(j-2)}-u^{(j-1)}\tilde{u}^{(j-2)})\Big)\notag\\
&=&\overset{J}{\underset{j=1}{\sum}}(u^{(j)}u^{(j-1)}\tilde{u}^{(j+1)}-u^{(j)}\tilde{u}^{(j-1)}u^{(j-2)})\notag\\
&=&\overset{J}{\underset{j=1}{\sum}}(u^{(j-1)}u^{(j-2)}\tilde{u}^{(j)}-u^{(j+1)}\tilde{u}^{(j)}u^{(j-1)})\notag\\
&=&\overset{J}{\underset{j=1}{\sum}}\Big(\tilde{u}^{(j)}(u^{(j-1)}u^{(j-2)}-u^{(j+1)}u^{(j-1)})\Big)\notag\\
&=&-\langle B(u,u),\tilde{u}\rangle.
\end{eqnarray*}
\end{proof}

\begin{proof}[Proof of Proposition \ref{absorbing ball}]
Taking the Euclidean inner product of $u(t)$ with equation (\ref{0}) and using properties \ref{itm:a} and \ref{itm:b} we get
\begin{equation*}
\frac{1}{2}\frac{\textrm{d}|u|^2}{\textrm{d}t}=-|u|^2+\langle f,u\rangle.
\end{equation*}
Using Young's inequality for the last term gives
\begin{equation*}
\frac{\textrm{d}|u|^2}{\textrm{d}t}+|u|^2\leq JF^2.
\end{equation*}
Therefore, using Gronwall's lemma,
\begin{equation*}
|u(t)|^2\leq|u_0|^2e^{-t}+JF^2(1-e^{-t}),
\end{equation*}
and the result follows.
\end{proof}

\begin{proof}[Proof of Property \ref{p:2.2}]
The first part is automatic since, if $q:=Qu$, then  for all $j$ either $q^{(j-1)}=0$ or
$q^{(j-2)}=q^{(j+1)}=0$. Since $B(Qu,Qu)=0$ and
$B(\cdot,\cdot)$ is a bilinear operator we can write
\begin{eqnarray*}
B(u,u)&=&B(Pu+Qu,Pu+Qu)\notag\\
&=&B(Pu,Pu)+2B(Pu,Qu).
\end{eqnarray*}
Now using property \ref{itm:d}, and the fact that there is $c>0$ such that $|Pu|+2|Qu|\leq \frac{c}{2}|u|,$ 
\begin{eqnarray*}
|\langle B(u,u),\tilde{u}\rangle |&\leq & |B(u,u)||\tilde{u}|\notag\\
&\leq & |B(Pu,Pu)+2B(Pu,Qu)||\tilde{u}|\notag\\
&\leq & 2|Pu||\tilde{u}|(|Pu|+2|Qu|)\notag\\
&\leq & c|Pu||\tilde{u}||u|.
\end{eqnarray*}
\end{proof}

\begin{proof}[Proof of Theorem \ref{t:54}]
Define the error in the approximate solution as $\delta = m-v  = q - Qv$. Note that $Q\delta=\delta$. 
The error satisfies the following equation 
\begin{equation*}
Q\frac{\textrm{d}\delta}{\textrm{d}t}+Q\delta+Q\big(B(Pv+q,Pv+q)-B(v,v)\big)=0.
\end{equation*}
Splitting $v=Pv+Qv$ and noting, from Properties \ref{p:2.2}, 
that $B(Qv,Qv)=0$ and $B(q,q)=0$, yields
\begin{equation*}
\frac{\textrm{d}Q\delta}{\textrm{d}t}+Q\delta+2QB(Pv,Q\delta)=0.
\end{equation*}
Taking the inner product with $Q\delta$ gives
\begin{equation*}
\frac{1}{2}\frac{\textrm{d}|Q\delta|^2}{\textrm{d}t}+|Q\delta|^2+2\langle B(Pv,Q\delta),Q\delta\rangle=0.
\end{equation*}
Note that from the Properties \ref{p:2.1}, \ref{itm:c} and \ref{itm:f}, and
Property \ref{p:2.2}, we have
\begin{eqnarray*}
2\langle B(u,Q\delta),Q\delta\rangle &=& - \langle B(Q\delta,Q\delta),u\rangle\notag\\
&=& 0.
\end{eqnarray*}
Thus since $Q\delta =\delta$ we have 
\begin{equation*}
\frac{\textrm{d}|\delta|^2}{\textrm{d}t}+2|\delta|^2 = 0,
\end{equation*}
and so
\begin{equation*}
|\delta(t)|^2 = |\delta(0)|^2e^{-2t}.
\end{equation*}
As $t\rightarrow\infty$ the error $\delta(t)\rightarrow 0$.
\end{proof}

\begin{proof}[Proof of Theorem \ref{t:55}]
 From \eqref{eq:m} and \eqref{eq:z} 
\begin{eqnarray*}
\frac{\textrm{d}m}{\textrm{d}t}=\mathcal{F}(m)+\frac{1}{\eta}\Big(Pv+\epsilon P\frac{\textrm{d}w}{\textrm{d}t}-Pm \Big).
\end{eqnarray*}
Thus
\begin{equation*}
\frac{\textrm{d}m}{\textrm{d}t}=-m-B(m,m)+f+\frac{1}{\eta}P(v-m)+\frac{\epsilon}{\eta}P\frac{\textrm{d}w}{\textrm{d}t}.
\end{equation*}
The signal is given by
\begin{equation*}
\frac{\textrm{d}v}{\textrm{d}t}=-v-B(v,v)+f, 
\end{equation*}
and so the error $\delta = m-v$ satisfies 
\begin{equation*}
\frac{\textrm{d}\delta}{\textrm{d}t}=-\delta-2B(v,\delta)-B(\delta,\delta)-\frac{1}{\eta}P\delta+ \frac{\epsilon}{\eta}P\frac{\textrm{d}w}{\textrm{d}t}.
\end{equation*}
Lemma \ref{l3} below, Properties \ref{p:2.1} and It\'o's formula give
\begin{equation*}
\frac{1}{2}\textrm{d}|\delta|^2+\Big(1-\frac{c^2K\eta}{4}\Big)|\delta|^2{\textrm{d}t}\leq \frac{\epsilon}{\eta}\langle P\textrm{d}w, \delta\rangle + \frac{J}{3}\frac{\epsilon^2}{\eta^2}{\textrm{d}t}.
\end{equation*}
Integrating and taking expectations
\begin{equation*}
\frac{\textrm{d}\mathbb{E}{|\delta|}^2}{\textrm{d}t}\le -\lambda\mathbb{E}{|\delta|}^2+\frac{2J\epsilon^2}{3\eta^{2}}.
\end{equation*}
Use of the Gronwall inequality gives the desired result.
\end{proof}

\vspace{0.1in}
We now turn to discrete-time data assimilation, where the following lemma plays an important role:
\begin{lemma}\label{lm1}
Consider the Lorenz '96 model \eqref{0} with $F>0$ and $J\geq3.$ Let $v$ and $u$ be two solutions in $[t_k,t_{k+1}),$ with $v(t_k)\in \mathcal{B}.$  Then there exists a $\beta \in \bbR$ such that
$$|u(t)-v(t)|^2\le |u(t_k)-v(t_k)|^2e^{\beta(t-t_k)} \quad t\in[t_k,t_{k+1}).$$ 
\end{lemma}

\begin{proof}
Let $\delta=m-v.$ Then $\delta$ satisfies
\begin{equation}\label{4}
\frac{1}{2}\frac{\textrm{d}|\delta|^2}{\textrm{d}t}+|\delta|^2+2\langle B(v,\delta), \delta\rangle+\langle B(\delta,\delta), \delta\rangle=0
\end{equation}
so that, by Property \ref{p:2.1}, item \ref{itm:b},
\begin{equation*}
\frac{1}{2}\frac{\textrm{d}|\delta|^2}{\textrm{d}t}+|\delta|^2 - 2|\langle B(v,\delta), \delta\rangle| \le 0.
\end{equation*}
Using Properties \ref{p:2.1} {items} \ref{itm:d} and \ref{itm:f} gives
$|\langle B(v,\delta), \delta\rangle| \le  K^{\frac{1}{2}}|\delta|^2$, where $K$ is defined in Proposition \ref{absorbing ball}, so that
\begin{equation*}
\frac{1}{2}\frac{\textrm{d}|\delta|^2}{\textrm{d}t}\le (2K^{\frac{1}{2}}-1)|\delta|^2 .
\end{equation*}
Integrating the differential inequality gives
\begin{equation}\label{2}
|\delta(t)|^2\le |\delta(t_k)|^2e^{\beta(t-t_k)}.
\end{equation}
\end{proof}

Note if $F < \frac{1}{2\sqrt{2} J}$ 
then $\beta =  2(2K^{\frac{1}{2}}-1)<0$ and the subsequent analysis
may be significantly simplified. Thus we assume in what follows that
$F \ge \frac{1}{2\sqrt{2} J}$ so that $\beta \ge 0.$
Lemma \ref{lm1} gives an estimate on the growth of the error in the forecast step. Our aim now is to show that this growth can be controlled by observing $Pv$ discretely in time. It will be required that the time $h$ between observations
is sufficiently small.

To ease the notation we introduce three functions that will be used in the proofs of Theorems \ref{p:2.2} and \ref{thm:dis}. Namely we define, for $t>0,$
\begin{equation}\label{A1}
A_1(t):=\frac{16K}{\beta}(e^{\beta t}-1)+\frac{4R_0^2}{2\beta}(e^{2\beta t}-1),
\end{equation}
\begin{equation}\label{A2}
B_1(t):=\frac{16c^2K^2}{\beta}\left[\frac{e^{\beta t}-e^{-t}}{\beta + 1}-(1-e^{-t})\right] +e^{-t}+ \frac{4c^2KR_0^2}{2\beta}\left[\frac{e^{2\beta t}-e^{-t}}{2\beta + 1}-(1-e^{-t})\right],
\end{equation}
and
\begin{equation}\label{B1}
B_2(t):=c^2K\{1-e^{-t}\}.
\end{equation}
Here and in what follows 
$c$, $\beta$ and $K$ are as in Property \ref{p:2.2}, Lemma \ref{lm1} and Proposition \ref{absorbing ball}.
We will use two different norms in $\mathbb{R}^J$ to prove the theorems that follow. In each case, the constant $R_0>0$ above quantifies the size of the initial error, measured in the relevant  norm for the result at hand.

\begin{proof}[Proof of Theorem \ref{t:56}]
Define the error $\delta = m-v.$ Subtracting equation \eqref{eq:en} 
from equation \eqref{8} gives
\begin{subequations}
\label{3}
\begin{align}
&\frac{\textrm{d}\delta}{\textrm{d}t}+\delta+2B(v,\delta)+B(\delta,\delta)=0,\quad t\in(t_k,t_{k+1}),\\
&\delta(t_k)=Q\delta(t_k^-)
\end{align}
\end{subequations}
{where $\delta(t_{k+1}^-):=\lim_{t\uparrow t_{k+1}}\delta(t)$ as defined in section \ref{ssec:D1}}. 
Notice that $B_1(0)=1$ and $B_1'(0)=-1$, so that there is $h^*>0$ with the property that $B_1(h) \in (0,1)$ for all $h \in (0,h^*].$ Fix any such assimilation time $h$ and denote $\gamma=B_1(h)\in(0,1).$
Let $R_0:=|\delta_0|$. We show by induction that, for every $k,$ $|\delta_k|^2\le \gamma^k R_0^2.$
We suppose that it is true for $k$ and we prove it for $k+1.$

Taking the inner product of $P\delta$ with the equation \eqref{3} gives
\begin{equation*}
\frac{1}{2}\frac{\textrm{d}|P\delta|^2}{\textrm{d}t}+|P\delta|^2+2\langle B(v,\delta), P\delta\rangle+\langle B(\delta,\delta), P\delta\rangle=0
\end{equation*}
so that, by Property \ref{p:2.1}, item \ref{itm:d},
\begin{equation*}
\frac{1}{2}\frac{\textrm{d}|P\delta|^2}{\textrm{d}t}+|P\delta|^2\le4|v||\delta||P\delta|+2|\delta|^2|P\delta|.
\end{equation*}
By the inductive hypothesis we have
$|\delta_k|^2\le R_0^2$ since $\gamma\in(0,1)$. Shifting the time origin by
setting $\tau:=t-t_k$ and using Lemma \ref{lm1} gives 
\begin{align}\label{eq:pdel1}
\frac{1}{2}\frac{\textrm{d}|P\delta|^2}{\textrm{d}\tau}+|P\delta|^2&\le4K^{\frac{1}{2}}|\delta||P\delta|+2|\delta_k|e^{\frac{\beta \tau}{2}}|\delta||P\delta| \notag\\ 
&\le4K^{\frac{1}{2}}|\delta||P\delta|+2R_0e^{\frac{\beta \tau}{2}}|\delta||P\delta|.
\end{align}
Applying Young's inequality to each term on the right-hand side we obtain 
\begin{equation}\label{diffequationpdelta}
\frac{\textrm{d}|P\delta|^2}{\textrm{d}\tau}\le 16 K|\delta|^2+4 R_0^2e^{\beta \tau}|\delta|^2.
\end{equation}
Integrating from $0$ to $s$, where $s\in (0,h)$, gives
\begin{align}
|P\delta(s)|^2 &\le A_1(s)|\delta_k|^2. \label{5}
\end{align}

Now again consider the equation \eqref{4} using Property \ref{p:2.1} {item} \ref{itm:f}
to obtain
\begin{equation*}
\frac{1}{2}\frac{\textrm{d}|\delta|^2}{\textrm{d}\tau}+|\delta|^2 - |\langle B(\delta,\delta), v\rangle| \le 0.
\end{equation*}

Using Property \ref{p:2.2} and Young's inequality yields
\begin{eqnarray}\label{bounddelta}
\frac{1}{2}\frac{\textrm{d}|\delta|^2}{\textrm{d}\tau}+|\delta|^2 &\le & c|v||\delta ||P\delta | \notag\\
&\le & cK^{\frac{1}{2}}|\delta ||P\delta | \notag\\
&\le & \frac{|\delta |^2}{2}+\frac{c^2K}{2}|P\delta |^2.
\end{eqnarray}
Employing the bound \eqref{5} then gives
\begin{equation*}
\frac{\textrm{d}|\delta|^2}{\textrm{d}\tau}+|\delta |^2 \le \left(\frac{16c^2 K^2}{\beta} (e^{\beta \tau} -1) +\frac{ 4c^2K R_0^2}{2\beta}(e^{2\beta \tau}-1)\right)|\delta_k|^2.
\end{equation*}
Therefore, upon using Gronwall's lemma, 
\begin{equation*}
|\delta(s)|^2 \le B_1(s)|\delta_k|^2.
\end{equation*}
It follows that 
$$|\delta_{k+1}|^2\le\gamma|\delta_k|^2\le \gamma^{k+1}R_0^2,$$
and the induction (and hence the proof) is complete.
\end{proof}

\begin{proof}[Proof of Theorem \ref{thm:dis}]
We define the error process $\delta(t)$ as follows:
\begin{equation}\label{deldef}
 \delta(t)=\left\{ \begin{array}{c}
\delta_k:=m_k-v(t)\quad \textrm{if}\:\: t=t_k\\
\Psi(m_k,t-t_k)-v(t)\quad \textrm{if}\:\: t\in(t_{k},t_{k+1}).\\
\end{array}\right.
\end{equation}

Observe that $\delta$ is discontinuous at times
$t_k$ which are multiples of $h$, since $m_{k+1} \ne
\Psi(m_k;h).$ 
Subtracting \eqref{eq:ve} from \eqref{eq:me} we obtain
\begin{align}
{\delta_{k+1}}&=\delta(t_{k+1})=\left(\frac{\eta}{1+\eta}P+Q\right) \delta(t^-_{k+1})+\frac{1}{1+\eta}{\nu_{k+1}}, \label{89}\\
\rev{P\delta_{k+1}} &\rev{= \frac{\eta}{1+\eta}P \delta(t^-_{k+1})+\frac{1}{1+\eta}\nu_{k+1},} \label{pdelta}
\end{align}
{where $\delta(t_{k+1}^-):=\lim_{t\uparrow t_{k+1}}\delta(t)$ as defined 
above and in section \ref{ssec:D1}}.

Let $A_1(\cdot)$, $B_1(\cdot)$ and $B_2(\cdot)$ be as in (\ref{A1}, \ref{A2}, \ref{B1}), and set 
\begin{align*}
M_1(t):=&\frac{2\eta}{1+\eta}\sqrt{A_1(t)}+\sqrt{B_1(t)},\\
M_2(t):=&\frac{2\eta}{1+\eta}+\sqrt{B_2(t)}.
\end{align*}
Since $A_1(0)=0, B_1(0)=1,$  $B_2(0)=0$ and
$$\frac{d}{dt}\sqrt{B_1(t)}\Big|_{t=0}=-1/2<0$$ 
it is possible to find $h,\eta>0$ small such that 
$$M_2(h)<M_1(h)=:\alpha<1.$$
Let $R_0=\|\delta_0\|$.  We show by induction that for such $h$ and $\eta,$ and provided that $\epsilon$ is small enough so that
$$\alpha R_0+2\epsilon<R_0,$$
we have that $\|\delta_k\|\le R_0$ for all $k.$ Suppose for induction that it is true for $k.$ 
Then $|\delta_k|\le \|\delta_k\|\le R_0$ and we can apply (after shifting time as before)  Lemma \ref{lem:dis} below to obtain that
$$|P\delta(\rev{t_k + t}))|\le\sqrt{A_1(t)|\delta_k|^2+|P\delta_k|^2}\le \sqrt{A_1(t)}|\delta_k|+|P\delta_k|$$
and
$$|\delta(\rev{t_k + t})|\le\sqrt{B_1(t)|\delta_k|^2+B_2(t)|P\delta_k|^2}\le \sqrt{B_1(t)}|\delta_k|+\sqrt{B_2(t)}|P\delta_k|.$$
Therefore, \rev{combining \eqref{89} and \eqref{pdelta}, and then using the two previous inequalities, we obtain that}
\begin{align*}
|P\delta_{k+1}|+|\delta_{k+1}| &\rev{\le \frac{2\eta}{1+\eta}|P\delta(t^-_{k+1})| + |\delta(t^-_{k+1})| + \frac{2}{1+\eta}|\nu_{k+1}|     }\\
&\le \left(\frac{2\eta}{1+\eta}\sqrt{A_1(h)}+\sqrt{B_1(h)}\right)|\delta_k| + \left(\frac{2\eta}{1+\eta}+\sqrt{B_2(h)}\right)|P\delta_k|+2\epsilon\\
&=M_1(h)|\delta_k|+M_2(h)|P\delta_k|+2\epsilon.
\end{align*}

Since $M_2(h)<M_1(h)=\alpha$ we deduce that
$$
\|\delta_{k+1}\|\le \alpha\|\delta_k\| + 2\epsilon,
$$
which proves \eqref{eq:bnd}.
Furthermore, the induction is complete, since
$$\|\delta_{k+1}\|\le \alpha\|\delta_k\| + 2\epsilon\le \alpha R_0+2\epsilon\le R_0.$$
\end{proof}

\begin{lemma}\label{l3}
Let $v\in\mathcal{B}.$ Then, for any $\delta,$
\begin{equation*}
\langle \delta + 2B(v,\delta)+B(\delta,\delta)+\frac{1}{\eta}P\delta, \delta\rangle \geq \Big(1-\frac{c^2K\eta}{4}\Big)|\delta|^2.
\end{equation*}
\end{lemma}

\begin{proof}
Use of Property \ref{p:2.1}, items \ref{itm:c} and \ref{itm:f},
together with Property \ref{p:2.2}, shows that
\begin{eqnarray*}
\langle \delta + 2B(v,\delta)+B(\delta,\delta)+\frac{1}{\eta}P\delta, \delta\rangle &=& |\delta|^2+2\langle B(v,\delta),\delta\rangle + \langle B(\delta,\delta),\delta\rangle+\langle \frac{1}{\eta}P\delta, \delta\rangle\notag\\
&= & |\delta|^2-\langle B(\delta,\delta),v\rangle + \langle \frac{1}{\eta}P\delta, \delta\rangle\notag\\
&\geq & |\delta|^2-cK^{\frac{1}{2}}|\delta||P\delta| + \frac{1}{\eta}|P\delta|^2\notag\\
&\geq & |\delta|^2-\frac{\theta|\delta|^2}{2}-\frac{c^2K|P\delta|^2}{2\theta} + \frac{1}{\eta}|P\delta|^2.
\end{eqnarray*}
Now choosing $\theta = \frac{c^2K\eta}{2}$ establishes the claim.

\end{proof}

\vspace{0.1in}

\begin{lemma}\label{lem:dis}
In the setting of Theorem \ref{thm:dis}, for $t\in [0,h)$ and $R_0:=\|\delta_0\|$ we have 
\begin{equation}\label{boundpdelta}
|P\delta(t)|^2\le A_1(t)|\delta_0|^2 + |P\delta_0|^2
\end{equation}
and
\begin{equation}\label{bounddeltalemma}
|\delta(t)|^2 \le B_1(t)|\delta_0|^2 + B_2(t)|P\delta_0|^2,
\end{equation}
where the error $\delta$ is defined as in \eqref{deldef} and $A_1, B_1$ and 
$B_2$ are given by (\ref{A1}, \ref{A2}, \ref{B1}).  
\end{lemma}

\begin{proof}
As in equation \eqref{diffequationpdelta} we have
$$
\frac{\textrm{d}|P\delta|^2}{\textrm{d}t}\le 16 K|\delta|^2+4 R_0^2e^{\beta t}|\delta|^2.
$$
On integrating from $0$ to $t$ as before, and noting that now $P\delta_0\neq 0$ in general, we obtain
$$
|P\delta(t)|^2 \le \left( \frac{16 K}{\beta} \{ e^{\beta t} -1\} +\frac{ 4 R_0^2}{2\beta}\{e^{2\beta t}-1\}\right)|\delta_0|^2+|P\delta_0|^2,
$$
which proves \eqref{boundpdelta}.

For the second inequality recall the bound \eqref{bounddelta}
$$
\frac{1}{2}\frac{\textrm{d}|\delta|^2}{\textrm{d}t}+|\delta|^2 \le  \frac{|\delta |^2}{2}+\frac{c^2K}{2}|P\delta |^2,
$$
and combine it with \eqref{boundpdelta} to get

$$
\frac{\textrm{d}|\delta|^2}{\textrm{d}t}+|\delta |^2 \le \left( \frac{16c^2 K^2}{\beta} \{ e^{\beta t} -1\} +\frac{ 4c^2K R_0^2}{2\beta}\{e^{2\beta t}-1\}\right)|\delta_0|^2+c^2K|P\delta_0|^2.
$$
Applying Gronwall's inequality yields \eqref{bounddeltalemma}.
\end{proof}

\end{document}